\documentclass[reqno]{amsart}
\usepackage{cmbright}
\usepackage[hidelinks]{hyperref}
\usepackage{enumitem}
\usepackage{tikz,doi}
\usepackage[utf8]{inputenc}
\usepackage{amsmath}
\usepackage{amssymb} 
\usepackage[numbers]{natbib}
\usepackage{prettyref}

\theoremstyle{plain}
\makeatletter
\newtheorem*{thm@P}{Theorem \pipp@}

\makeatother

\newenvironment{pf}{\begin{proof}}{\end{proof}}

\numberwithin{equation}{section}

\newtheorem{thm}{Theorem}[section]
\newtheorem*{thm*}{Theorem}
\newtheorem*{cor*}{Corollary}
\newtheorem*{defn*}{Definition}
\newrefformat{thm}{Theorem~\ref{#1}} 
\makeatletter 
\let\old@newtheorem\newtheorem 
\renewcommand{\newtheorem}[2]{\old@newtheorem{#1}[thm]{#2} 
	\newrefformat{#1}{#2~\ref{##1}}} 
\makeatother 

\newtheorem{prop}{Proposition}
\newtheorem{cor}{Corollary}

\newtheorem{lem}{Lemma}

\newtheorem{fact}{Fact}

\theoremstyle{definition}
\newtheorem{rem}{Remark}
\newtheorem{observation}{Observation}

\newtheorem{defn}{Definition}

\newtheorem{exa}{Example}
\newtheorem{example}{Example}

\newtheorem{notation}{Notation}

\newcommand{\Q}{\mathbb Q}
\newcommand{\R}{\mathbb R}
\newcommand{\N}{\mathbb N}

\newcommand{\Z}{\mathbb Z}

\newcommand{\e}{1_G}

\newcommand{\omin}[1]{{[#1]}^{\text{o-min}}}

\DeclareMathOperator{\dom}{dom}
\DeclareMathOperator{\Img}{Im}

\usepackage[textsize=tiny]{todonotes}

\newcommand{\sub}{\subseteq}

\newcommand{\Rarr}{\ensuremath{\Rightarrow}}

\newcommand{\CM}{{\mathcal M}}

\newcommand{\CY}{{\mathcal Y}}

\newcommand{\cal}[1]{\ensuremath{\mathcal{#1}}}

\title{Orthogonal decomposition of definable groups}
\author{Alessandro Berarducci}
\address{Universit\`a di Pisa,
Dipartimento di Matematica,
Largo Bruno Pontecorvo, 5,
56127 Pisa, Italy}
\email{alessandro.berarducci@unipi.it}
\author{Pantelis E.\ Eleftheriou}
\address{School of Mathematics, University of Leeds, Leeds LS2 9JT, United Kingdom}
\email{p.eleftheriou@leeds.ac.uk}
\author{Marcello Mamino}
\address{Universit\`a di Pisa,
Dipartimento di Matematica,
Largo Bruno Pontecorvo, 5,
56127 Pisa, Italy}
\email{marcello.mamino@unipi.it}
\date{January 3, 2021}
\thanks{Berarducci and Mamino have been partially supported by the Italian research
project PRIN~2017, ``Mathematical logic: models, sets, computability'',
Prot.~2017NWTM8RPRIN. Eleftheriou has been partially supported by 2017NWTM8RPRIN, a Zukunftskolleg Research Fellowship (Konstanz) and an EPSRC Early Career Fellowship (EP/V003291/1) (Leeds).
}

\begin{document}
\begin{abstract} Orthogonality in model theory captures the idea of absence of non-trivial interactions between definable sets. We introduce a somewhat opposite notion of cohesiveness, capturing the idea of interaction among all parts of a given definable set.
A cohesive set is indecomposable, in the sense that if it is internal to the product of two orthogonal sets, then it is internal to one of the two.
We prove that a definable group in an o-minimal structure is a product of cohesive orthogonal subsets.
 If the group has dimension one, or it is definably simple, then it is itself cohesive. As an application, we show that an abelian group definable in the disjoint union of finitely many o-minimal structures is a quotient, by a discrete normal subgroup, of a direct product of locally definable groups in the single structures.
\end{abstract}

\maketitle
\tableofcontents

\section{Introduction}
Considering a group $G$ interpretable in the disjoint union of finitely
many structures ${\mathcal X}_1,\ldots, {\mathcal X}_n$ (seen as a multi-sorted structure as in Definition \ref{defn:disjoint}), one may ask whether $G$ can be understood in
terms of groups definable in the individual structures. One may, for instance, ask whether $G$ is definably isomorphic to a quotient, modulo a finite normal subgroup $\Gamma$, of a direct product $G_1\times \ldots \times G_n$, where $G_i$ is a  group definable in ${\mathcal X}_i$. This, however, is not true in general: a counterexample is provided by \cite[Example 1.2]{Berarducci2013c} (a torus obtained from two orthogonal copies of $\R$ and a lattice generated by two vectors in generic position). After a talk by the first author at the Oberwolfach workshop ``Model Theory: Groups, Geometry, and Combinatorics'' (2013), Hrushovski suggested that a result of the above kind would require to pass to the locally definable category. So the natural conjecture would be that, if $G$ is as above, there is a locally definable isomorphism $G \cong G_1\times \ldots \times G_n/\Gamma$, where $G_i$ is a locally definable group in $\mathcal{X}_i$ and $\Gamma$ is a compatible {\it discrete} subgroup (i.e.\ a subgroup which intersects every definable set at a finite set). Here we establish the conjecture under the additional assumption that the structures ${\mathcal X}_i$ are o-minimal and $G$ is abelian. More precisely, we prove the following result.

\begin{thm*}[\ref{thm:lattice2}]
	Let $G$ be a definable abelian group in the disjoint union of finitely many o-minimal structures ${\mathcal X}_1, \ldots, {\mathcal X}_n$. Then there is a locally definable homomorphism $$G \cong G_1\times \ldots \times G_n /\Gamma,$$ where $G_i$ is a locally definable group in ${\mathcal X}_i$ and $\Gamma$ is a compatible locally definable discrete subgroup of $G_1\times \ldots \times G_n$.
\end{thm*}

We will deduce Theorem \ref{thm:lattice2} from Theorem \ref{thm:ominimal} below, which is interesting in itself and holds for non-abelian groups as well. To state the theorem we need to recall the model-theoretic notion of orthogonality. Given definable sets $X_1, \ldots, X_n$ in a structure $\CM$, we say that $X_1,\ldots, X_n$ are {\it orthogonal} if, for all $k_1, \ldots,  k_n \in \N$, any definable subset of $X^{k_1} \times  \ldots \times  X^{k_n}$ is a finite union of sets of the form $A_1\times \ldots \times A_n$, where $A_i$ is a definable subset of $X_i^{k_i}$. Let us also recall that a definable set $X$ is \textit{internal} to a definable set $Y$ if there is $m\in \N$ and a definable surjective map from $Y^m$ to $X$.

\begin{thm*}[\ref{thm:ominimal}] Assume $\CM$ is an o-minimal structure. Let $X_1,\dots,X_n$ be orthogonal sets definable in $\CM$ and  $G$  a group definable in $\CM$. If $G$ is internal to the product $X_1\times \dots \times X_n$, then $G$ is a product $$G=A_1 \dots A_n$$ of definable subsets $A_1,\dots,A_n$, where $A_i$ is internal to $X_i$, for  $i=1,\dots,n$.
\end{thm*}

To deduce Theorem \ref{thm:lattice2} from Theorem \ref{thm:ominimal}, we take for $\CM$ the o-minimal structure consisting of the concatenation of the structures ${\mathcal X}_i$ separated by single points. The structures ${\mathcal X}_i$ would then be orthogonal within $\CM$, so we can apply Theorem \ref{thm:ominimal} and take as $G_i$ an isomorphic copy of the subgroup of $G$ generated by $A_i$. The only delicate point is to show that $G_i$ is locally definable in ${\mathcal X}_i$, but this is not difficult.

It is worth stressing that Theorem \ref{thm:ominimal} holds for an arbitrary o-minimal structure $\CM$ and in particular we do not assume that $\CM$ has a group structure. This is important for the application to Theorem \ref{thm:lattice2} because the concatenation of o-minimal structures does not have a group structure even if the single structures do.

The proof of Theorem \ref{thm:ominimal} uses a number of deep results about groups definable in o-minimal structures, such as the solution of Pillay's Conjecture and Compact Domination Conjecture (see \cite{Pillay2004a} and \cite{Hrushovski2007a} for the definitions).  It is, however, conceivable that the theorem could be extended far beyond the o-minimal context: indeed we have no counterexample even if $\CM$ is allowed to be a completely arbitrary structure. Theorem \ref{thm:NIP} provides a partial result in this direction, when $\CM$ is NIP and $G$ is compactly dominated and abelian.

Before stating our final result, we
notice that the sets $A_1,\ldots, A_n$ in Theorem \ref{thm:ominimal} are orthogonal, so we may ask whether, for a group $G$ definable in an arbitrary o-minimal structure $\CM$, there is always a natural way to generate it as a product $G=A_1\ldots A_n$ of orthogonal definable subsets $A_i$. Of course, there is always the trivial solution with $n=1$ and $A_1 = G$, but we would like each $A_i$ to be in some sense ``minimal". To this aim, we introduce the following model-theoretic notions. Let Z be a set definable in an arbitrary structure $\CM$. We say that:
	\begin{itemize}
			\item  $Z$ is \textbf{indecomposable} if whenever $Z$ is internal to the cartesian product of two orthogonal sets, then $Z$ is internal to one of the two.
		\item  $Z$ is {\bf cohesive} if whenever two definable sets are non-orthogonal to $Z$, they are non-orthogonal to each other.
		\end{itemize}
A cohesive set $Z$ is  indecomposable (Proposition \ref{prop:int-cohesive}(2)), but it can be shown that the converse fails (we thank Rosario Mennuni for this observation). We show that cohesive sets have nice model-theoretic properties which we are not able to prove for indecomposable sets. In particular, a set internal to a cohesive set is cohesive and the cartesian product of two cohesive sets is cohesive. The intuition is that the various parts of a cohesive set interact in a non-trivial way, so in particular a cohesive set cannot contain two  orthogonal infinite definable subsets. Our final result is the following.

\begin{thm*}[\ref{thm:main2}]
	Let $G$ be a  group interpretable in an o-minimal structure $\CM$. Then there are cohesive orthogonal definable sets $A_1,\ldots, A_n$, such that $G=A_1\ldots A_n$.
\end{thm*}
In the setting of Theorem \ref{thm:main2} we call the tuple $A_1,\ldots, A_n$ an {\bf orthogonal decomposition} of $G$, while in Theorem \ref{thm:ominimal} the tuple $A_1,\ldots, A_n$ is called a {\bf decomposition} of $G$ with respect to $X_1,\ldots, X_n$. One can loosely make the following analogy: a decomposition is like a factorization, while an orthogonal decomposition is like a factorization into primes.

If $G$ is infinite, we can choose each $A_i$ in Theorem \ref{thm:main2} to be infinite, and in this case the number $n$ is an invariant of $G$ up to definable isomorphism.  Indeed, if $G=B_1\ldots B_m$ is another decomposition of $G$ as a product of orthogonal cohesive infinite definable subsets, then $m=n$ and each $B_i$ is bi-internal to a single $A_j$. We may call the invariant $n$ the {\bf dimensionality} of $G$ (not to be confused with the dimension of $G$). In this terminology, the unidimensional groups in the sense of \cite[Claim 1.26]{Peterzil2000} have dimensionality $1$.

We can show that if $G$ has dimension one, or it is definably simple, then $G$ is itself cohesive, so these groups have dimensionality one.

For the proof of Theorem \ref{thm:main2}, we  need both Theorem \ref{thm:ominimal} and the results from \cite{Ramakrishnan2014}. In particular, we need that for every group $G$ interpretable in an o-minimal structure, there is an injective definable map $f$ (not necessarily a morphism) from $G$ to the cartesian product of finitely many $1$-dimensional definable groups (\cite[Theorem 3]{Ramakrishnan2014}).

\subsection*{Related work}
Groups definable in the disjoint union of orthogonal structures have already been considered in \cite{Berarducci2013c,Wagner2016}. The original motivation comes from the model theory of group extensions \cite{Berarducci2010c,Zilber2006}. For instance, the universal cover of a group definable in an o-minimal expansion of the field $\R$ is definable in the disjoint union $\R \sqcup \Z$, where $\Z$ has only the additive structure \cite{Hrushovski2011}. From a model-theoretic point of view, $(\Z,+)$ is an example of a superstable structure of finite and definable Lascar rank. In \cite{Berarducci2013c}, it is shown that if a group $G$ is definable in the disjoint union of an arbitrary structure $\cal R$ and a superstable structure $\cal Z$ of finite and definable Lascar rank, then $G$ is  an extension of a group internal to $\cal R$ by a group internal to $\cal Z$.  In \cite{Wagner2016},  Wagner weakened the superstability assumption to assuming only that  $\cal Z$ is simple. The simplicity assumption cannot be entirely removed, or replaced by an o-minimality one: quotients by a lattice of a  product of copies of $\R$ provide examples of definable groups which may not have infinite definable subgroups internal to any of the copies (\cite[Example 1.2]{Berarducci2013c}). However, as Theorem \ref{thm:ominimal} shows, the o-minimality assumption allows for another, in fact more symmetric analysis in terms of generating subsets instead of subgroups. Theorem \ref{thm:main2} can be seen as a continuation of the work in \cite{Ramakrishnan2014}.

The possibility of extending the results beyond the o-minimal context
raises a number of questions, which are included in Section~\ref{sec:questions}.

\subsection*{Structure of the paper.}
In Sections \ref{sec:orth} -- \ref{sec:splitting}, we introduce and study
the key notions of this paper. In particular we prove, for definable sets $X_1,\ldots, X_n$ in an arbitrary structure, that $X_1,\ldots, X_n$ are orthogonal if and only if they are pairwise orthogonal. We then study indecomposable and cohesive sets and
establish that $1$-dimensional groups definable in o-minimal structures
are cohesive (Theorem \ref{thm:dim1cohesive}). The proofs of Theorems
\ref{thm:ominimal} and \ref{thm:main2}  then proceed in several steps. In
Section \ref{sec:cdom}, we recall the basics of compact domination for NIP
structures, which we employ in Section \ref{sec:NIP} to prove Theorem
\ref{thm:ominimal} for compactly dominated abelian NIP groups that are
\emph{contained} in $X_1\times \dots \times X_n$ (Theorem \ref{thm:NIP}). In  Section
\ref{sec:omin1}, we specialize in o-minimal structures  and prove Theorem
\ref{thm:ominimal} for definably compact abelian groups \emph{internal} to
$X_1\times \dots \times X_n$. In Section \ref{sec:omin2}, we employ the
rich machinery available for groups definable in o-minimal structures, and
conclude the full Theorem \ref{thm:ominimal}. Together with the results
from \cite{Ramakrishnan2014}, we then establish Theorem \ref{thm:main2}.
In the final part of the paper we prove Theorem \ref{thm:lattice2}.

\subsection*{Notation}
Throughout this paper, we work in a first-order structure $\cal M$. By ``definable'' we mean definable in $\CM$, with parameters. Unless stated otherwise, $X, Y, Z$ denote definable sets. By convention, $X\times Y^{0}=X$.

 We assume familiarity with the basics of o-minimality,
as in~\cite{vdDries1998}. We also assume familiarity with the definable
manifold topology of definable groups~\cite[Proposition 2.5]{Pillay1988}.
All topological notions for
definable groups, such as connectedness and definable compactness, are taken
with respect to this group topology.

\section{Orthogonality}\label{sec:orth}

In this section we work in an arbitrary structure \cal M. We recall the notions of orthogonality and internality, and prove some basic facts.

\begin{defn}
Given definable sets $X_1,\dotsc,X_n$, a
$(X_1,\dotsc,X_n)${\bf-box} is
a definable set of the form~$U_1\times\dotsb\times U_n$ where
$U_i\subseteq X_i$ for every
$i=1,\dotsc,n$.
When clear from the
context, we omit the prefix $(X_1,\dotsc,X_n)$- in front of ``box''.
\end{defn}

\begin{defn}\label{def:orth} Let $X_1,\ldots, X_n$ be definable sets. We say that $X_1,\ldots, X_n$ are  {\bf orthogonal} if, for every $k_1,\ldots, k_n \in \N$, every definable subset $S$ of $X_1^{k_1}\times \ldots \times X_n^{k_n}$ is the
union of finitely many $(X_1^{k_1},\ldots, X_n^{k_n})$-boxes.
\end{defn}

An example of the notion of orthogonal sets is provided by the following definition. 

\begin{defn}\label{defn:disjoint}
	Given finitely many structures ${\mathcal X}_1, \ldots, {\mathcal X}_n$, their {\bf disjoint union} $\bigsqcup_i {\mathcal X}_i$ is the multi-sorted structure with a sort for each ${\mathcal X}_i$ and whose basic relations are the definable sets in the single structures ${\mathcal X}_i$.
Notice that if $X_i$ is the domain of ${\mathcal X}_i$, then $X_1,\ldots, X_n$ are orthogonal in $\bigsqcup_i {\mathcal X}_i$.
\end{defn}

The definition of orthogonality can be rephrased in terms of types using
the following remark of Wagner. We include a proof to facilitate comparison with similar notions of orthogonality in the model-theoretic literature (see \cite{Tent2012}),
but we will not need this fact.

\begin{rem}[{\cite[Remark 3.4]{Wagner2016}}]\label{rem:orthogonal}
	Let $X_1, \dots, X_n$ be definable sets. The following conditions are equivalent:
	\begin{enumerate}
		\item Every definable subset of $X_1 \times \dots \times
		X_n$ is a finite union of $(X_1,\dotsc,X_n)$-boxes.
		\item For every type $p(x_1,\dots, x_n)$ over $\CM$ with $p(x_1,\dots, x_n) \vdash  X_1\times \dots \times X_n$, we have that $p_1(x_1)\cup \dots \cup p_n(x_n) \vdash p(x_1,\dots, x_n)$ where $p_i$ is the $i$-th projection of $p$.
	\end{enumerate}
\end{rem}

\begin{pf} To simplify notation we assume $n=2$ and write $X,Y$ for $X_1,X_2$.
	
	Assume (1). Let $p(x,y)$ be a type over $\CM$ concentrated on
	$X\times Y$.  Let $\varphi(x,y)$ be a formula over $\CM$
	defining a subset of $X\times Y$. The set $Z$ defined by
	$\varphi$ is a finite union of boxes~$U_i\times V_i$. It follows that $p_1(x) \cup p_2(y) \vdash p(x,y)$.
	
	Assume (2).
Let $Z\subseteq X \times Y$ be definable. We must
	show that $Z$ is a finite union of boxes~$U_i \times V_i$. For any type $p(x,y)$ containing the defining formula of $Z$,  we have that $p_1(x) \cup  p_2(y) \vdash  (x, y) \in Z$.
	By compactness, there are $\varphi^p_1(x)\in p_1(x)$ and $\varphi^p_2(y) \in p_2(y)$, such that $\varphi^p_1(x) \land \varphi^p_2(y) \vdash (x,y)\in Z$.
	Again by compactness, $(x,y)\in Z$ is equivalent to a finite disjunction of formulas of the form $\varphi^p_1(x) \land \varphi^p_2(y)$ (if not we reach a contradiction considering a type containing the defining formula of $Z$ and the negation of all the formulas $\varphi^p_1(x) \land \varphi^p_2(y)$). It follows that $Z$ is a finite union of boxes $U_i\times V_i$ as desired.
\end{pf}

The following fact follows at once from the definition.

\begin{prop}\label{prop:box}
Let $X$, $Y$, $Z$ be definable sets.
Suppose that, for every positive integer~$n$, all definable subsets
of~$X^n\times Y$ are a finite union of $(X^n,Y)$-boxes,
and all definable subsets of~$X^n\times Z$ are finite
unions of $(X^n,Z)$-boxes. Then, for all~$n$, every definable subset
of~$X^n\times Y\times Z$ is a finite union of~$(X^n,Y\times Z)$-boxes.
\end{prop}
\begin{pf}
Let $S$ be a definable subset of~$X^n\times Y\times Z$.
Given $z\in Z$, consider the fiber $S_z \subseteq
X^n\times Y$ consisting of the pairs $(x,y)$ such that $(x,y,z)\in S$. Let
$E_z \subseteq X^n\times X^n$ be the following equivalence relation:
\[ a\,E_z\,b \quad \Longleftrightarrow \quad \forall y\in Y \;\; (a,y)\in S_z
\longleftrightarrow (b,y) \in S_z\;\text{.}\]
Then $\{E_z \mid
z\in Z\}$ is  a family of subsets of $X^n\times X^n$ indexed by $Z$, so by
the hypothesis applied to $X^{2n}\times Z$, it is finite.
By the hypothesis on $X^n\times Y$,
each equivalence relation $E_z$ has finitely many equivalence
classes; in fact, $S_z$ is a finite union of $(X^n,Y)$-boxes and the
equivalence classes of $E_z$ are the atoms of the boolean algebra
generated by the projections of these boxes on the component~$X^n$.
It follows that $E = \bigcap_z E_z \subseteq X^n\times X^n$
is again an equivalence relation with
finitely many classes. On the other hand, for  $a,b\in X^n$,
\[ a\,E\,b \quad \Longleftrightarrow \quad
\forall y\in Y, z\in Z \;\; (a,y,z)\in S \longleftrightarrow
(b,y,z)\in S\;\text{.}\]
We have thus proved that
there are finitely
many subsets of the form $\pi_{Y\times Z}(\pi_{X^n}^{-1}(x)\cap S)$
with $x\in X^n$, which is desired result.
\end{pf}

\begin{cor}\label{cor:both}
Let $X$, $Y$, $Z$ be definable sets.
	If $X$ is orthogonal to both $Y$ and $Z$, then $X$ is orthogonal to $Y\times Z$.
\end{cor}

\begin{cor}
	Suppose $X_1, \ldots, X_n$ are  pairwise orthogonal definable sets. Then $X_1,\ldots, X_n$ are orthogonal.
\end{cor}
\begin{pf}
	If suffices to show that each $X_i$ is orthogonal to the product of the other sets $X_j$. This follows by \ref{cor:both} and induction on $n$.
\end{pf}

We shall need the following result.

\begin{cor}\label{cor:weak-ortho}
Let $X$ and $Y$ be definable sets. Then $X$ and~$Y$ are orthogonal
if and only if for every positive integer~$n$, all definable subsets
of~$X^n\times Y$ are finite unions of $(X^n,Y)$-boxes.
\end{cor}
For the main results of this paper we make no saturation assumptions on the ambient structure $\CM$. It is however worth mentioning that under a saturation assumption we can strengthen Corollary \ref{cor:weak-ortho} as follows.

\begin{prop}\label{prop:XY-box}
	If $\CM$ is $\aleph_0$-saturated, then $X$ and~$Y$ are orthogonal
	if and only if all definable subsets
	of~$X\times Y$ are finite unions of $(X,Y)$-boxes. The same conclusion holds without saturation provided $\CM$ is o-minimal, or more generally if $\CM$ eliminates the quantifier $\exists^\infty$. 
\end{prop}
\begin{pf}
	Suppose that all definable subsets
	of~$X\times Y$ are finite unions of $(X,Y)$-boxes. Let $S$ be a definable subset of $X^n \times Y$. It suffices to show that $S$ is a finite union of $(X^n,Y)$-boxes (by Corollary \ref{cor:weak-ortho}). We reason by induction on $n$. The case $n=1$ holds by the assumptions. Assume $n>1$. For $t\in X$, consider the set
$$S_t = \{ (x,y) \mid (t,x,y) \in S\} \subseteq X^{n-1}\times Y.$$ By induction, $S_t$ is a finite union of $(X^{n-1},Y)$-boxes.
	Let $R_t\subseteq Y\times Y$ be the equivalence relation defined by $uR_t v \iff (\forall x\in X^{n-1})\; (x,u)\in S_t \leftrightarrow (x,v)\in S_t$.  Note that $R_t$ has finitely many equivalence classes. By saturation (or elimination of $\exists^\infty$), there is a uniform bound $k\in \N$, such that for all $t\in X$, there are at most $k$ equivalence classes modulo $R_t$.
	
	We claim that there is a finite subset $A$ of $Y$ such that for all $t\in X$ each equivalence class of $R_t$ intersects $A$. To this aim we prove, by induction on $j\leq k$, that there is a finite subset $A_j$ of $Y$ such that for all $t\in Y$ there are at most $k-j$ equivalence classes of $R_t$ which do not intersect $A_j$. Clearly we can take $A_0$ to be the empty set. Let us construct $A_{j+1}$. Consider the definable family $(Q_{t})_{t\in X}$ where $Q_{t}\subseteq Y$ is the union of the $R_t$-equivalence classes which intersect $A_j$. By the assumptions the family  $(Q_{t})_{t\in X}$ contains finitely many distinct subsets $Q_1,\ldots, Q_l$ of $Y$. We obtain $A_{j+1}$ adding to $A_j$ a point in $Y \setminus Q_i$ for each $i=1,\ldots, l$ such that $Q_i \neq Y$. The claim is thus proved taking $A = A_k$.
	
	Now we claim that the equivalence relation $R = \bigcap_{t\in X} R_t$ has finitely many equivalence classes. Indeed, given $a\in A$ consider the family $(P_{t,a})_{t\in X}$ where $P_{t,a}\subseteq Y$ is the $R_t$-equivalence class of $a$. For each $a\in A$, by the assumption there are finitely many sets of the form $P_{t,a}$ for $t\in X$.  Each $R_t$-equivalence class is of the form $P_{t,a}$. Hence each $R$-equivalence class belongs to the finite boolean algebra generated by the sets $P_{t,a}$. The claim is thus proved.
	
	Now for each $R$-equivalence class $B\subseteq Y$ and for each $y_1,y_2\in B$, we have $\forall t\in X, \forall x \in X^{n-1} \; (t,x,y_1)\in S \iff (t,x,y_2) \in S$. Thus $S \cap \pi_Y^{-1}B$ is a box where $\pi_Y:X^n\times Y \to Y$ is the projection. Therefore $S$ is a finite union of $(X^n,Y)$-boxes.
\end{pf}

We now turn to the notion of internality.

\begin{defn}[{\cite[Lemma 10.1.4]{Tent2012}}]\label{def-int} Let \cal M be a structure.  Given two definable sets $X$ and $V$, we say that	$X$ is {\bf internal to $V$}, or
	\textbf{$V$-internal}, if there is a definable surjection from $V^n$ to $X$, for some $n$.
\end{defn}

\begin{fact}[{\cite[Lemma 2.2]{Berarducci2013c}}] \label{fact:family} Let $X$ and $Y$ be orthogonal definable sets and let $(T_s \mid s\in S)$ be a definable family of subsets of a $Y$-internal set $T$ indexed by an $X$-internal set $S$. Then the family contains only finitely many distinct sets $T_s \subseteq T$.
\end{fact}

\begin{prop}[{\cite[Proposition 2.3]{Berarducci2013c}}]\label{prop:finite-proj}
	Let $X$ and $Y$ be orthogonal definable sets with $|X|\geq 2$, and  $S$  a definable subset of $X\times Y$. Then $S$ is $X$-internal if and only if its projection onto $Y$ is finite.
\end{prop}

It is clear that if $X$ and $Y$ are infinite and one of them is internal to the other, then the two sets are non-orthogonal. Also, if one of $X$ and $Y$ is finite, then the two sets are orthogonal.

\begin{prop}\label{prop:intort}
	Let $X$ and $Y$ be orthogonal definable sets. If $U$ is internal to $X$ and $V$ is internal to $Y$, then $U$ and $V$ are orthogonal.
\end{prop}
\begin{pf}
	Let $S\subseteq U^{k}\times V^{l}$ be definable. By the assumption there are definable surjective maps $f:X^m\to U^{k}$ and $g:Y^n\to V^l$.Then, by orthogonality, $(f\times g)^{-1}(S)$ is a finite union $\bigcup_i A_i\times B_i$ of $(X^m, Y^n)$-boxes. So $S = \bigcup_i f(A_i) \times g(B_i)$ is a finite union of $(U^{k},V^{l})$-boxes.
\end{pf}

\section{Cohesive sets}\label{sec:cohesive}

In this section we work in an arbitrary structure~$\CM$, except for
Theorem~\ref{thm:dim1cohesive} where we assume o-minimality. We introduce and study the following two key notions, also mentioned in the introduction.

\begin{defn} \label{defn:cohesive} Let $Z$ be a definable set.
	\begin{itemize}
		\item $Z$ is \textbf{indecomposable} if for every orthogonal definable sets $X,Y$, if $Z$ is internal to $X\times Y$, then $Z$ is either internal to $X$ or internal to $Y$.
		\item $Z$ is {\bf cohesive} if all definable sets $X$ and $Y$ not orthogonal to $Z$ are not orthogonal to each other.
	\end{itemize}
\end{defn}

\begin{prop}\label{prop:int-cohesive} \mbox{}
	\begin{enumerate}
		\item 	A definable set internal to a cohesive set is cohesive.
		\item Any cohesive set is indecomposable.
	\end{enumerate}
\end{prop}

\begin{pf} (1) Let $X$ be internal to $Y$. If $A$ and $B$ are non-orthogonal to $X$,  then by Proposition \ref{prop:intort} they are non-orthogonal to $Y$. Thus if $Y$ is cohesive, so is $X$.
	
We prove (2). Let $Z$ be a cohesive
set and suppose that $Z$ is internal to $X\times Y$ where $X$ and $Y$ are
orthogonal. By definition there is $m\in \N$ and a surjective definable
map $f: X^m\times Y^m \to Z$. Since $X^m$ and $Y^m$ are orthogonal, for
the sake of our argument we can assume $m=1$. So we have a surjective
definable map $f:X\times Y \to Z$ and we need to show that $Z$ is
$X$-internal or $Y$-internal. For $x\in X$ and $y\in Y$, let $f_x(y) =
f(x,y) = f^y(x)$. The image of $f_x$ is $Y$-internal and the image of
$f^y$ is $X$-internal.
	These two images are then orthogonal (Proposition \ref{prop:intort}), so they cannot be both infinite by the hypothesis on $Z$.  It follows that either $\Img(f_x)$ is finite for all $x\in X$ or $\Img(f^y)$ is finite for all $y\in Y$. By symmetry let us assume that $\Img(f^y)\subseteq Z$ is finite for all $y$.
	Let $E_y\subseteq X^2$ be the equivalence relation defined by $xE_y x' \iff f(x,y) = f(x',y)$. Then $E_y$ is a definable equivalence relation on $X$ of finite index. Since $(E_y \subseteq X\times X \mid y\in Y)$ is a definable family of subsets of an $X$-internal set indexed by a $Y$-internal set, by orthogonality there are finitely many sets of the form $E_y$ for $y\in Y$ (Fact \ref{fact:family}). The intersection $E = \bigcap_{y\in Y}E_y$ is then again a definable equivalence relation of finite index on $X$. Let $x_1, \dots, x_k$ be representatives for the equivalence classes of $E$. Then $Z$ is the image of the restriction of $f$ to $\bigcup_{i\leq k} \{x_i\} \times Y$, and therefore it is $Y$-internal.
\end{pf}

\begin{prop}\label{prop:prod-cohesive}
	If $X$ and $Y$ are cohesive and non-orthogonal, then $X\times Y$ is cohesive.
\end{prop}
\begin{pf}
	Let $A$ and $B$ be non-orthogonal to $X\times Y$. We need to prove
	that $A$ and $B$ are non-orthogonal. By Corollary~\ref{cor:both} either $X$ or $Y$ is non-orthogonal to $A$. Similarly, either $X$ or $Y$ is non-orthogonal to $B$. There are four cases to consider, but by symmetry we can consider the following two cases.
	
	Case 1. $X$ is non-orthogonal to both $A$ and $B$.
	
	Case 2. $X$ is non-orthogonal to $A$, and $Y$ is non-ortogonal to $B$.
	
	In the first case by the cohesiveness of $X$ the sets $A$ and $B$ are non-orthogonal. In the second case, since $X$ and $Y$ are non-orthogonal and $Y$ is non-orthogonal to $B$, by the cohesiveness of $Y$ we conclude that $X$ is non-orthogonal to $B$, so we have a reduction to the first case.
\end{pf}

We now turn to groups definable in o-minimal structures.

\begin{thm}\label{thm:dim1cohesive}
Let $G$ be a definable group of dimension $1$ in an o-minimal structure
$\CM$. Then $G$ is cohesive.
\end{thm}
\begin{proof}
Let $A$ and $B$ be definable sets non-orthogonal to $G$. We need to prove
that $A$ and $B$ are not orthogonal. Let us first concentrate on $A$.
By Corollary \ref{cor:weak-ortho} there are $n\in \N$ and a definable
relation~$R\subseteq A^n \times G$
which is not a finite union of boxes.
Let $\mathcal P (G)$ be the power
set of $G$ and let $f\colon A^n \to \mathcal P(G)$ be defined by
$f(a) = \{g\in G \mid (a,g) \in R\}$.
Since $f(a)$ has dimension $\leq 1$, its boundary
$\delta f(a)$, closure minus interior,
is either empty or has dimension $0$. By the assumption on $R$, the image of $f$ is infinite, thus
also the image of $\delta f\colon A^n\to \mathcal P (G)$ is infinite, since
in a group of dimension $1$ there can only be finitely many definable
subsets with a given boundary.  Now recall that a set of dimension zero is
finite and a definable family of finite sets is uniformly finite. So there
is $k\in \N$ such that $\delta f(a)$ has at most $k$ elements for every $a\in
A^n$. By ordering the points of $\delta f(a)$ lexicographically,
we have a map $h\colon A^n \to G^k$ with infinite image. It follows that there is $i\leq
k$ such that $\pi_i \circ h\colon A^n \to G$ has infinite image, where
$\pi_i\colon G^k\to G$ is the projection onto the $i$-th component.  We have thus
proved that there is a definable map $f_A\colon A^n \to G$ with infinite image.
Similarly, there is $l\in \N$ and a definable map $f_B\colon B^l\to G$ with
infinite image. Infinite definable subsets of a one-dimensional group have
non-empty interior, thus, composing with a group translation, we can assume that
the two images have an infinite intersection. The relation $xQy :\iff
f_A(x) = f_B(y)$ then witnesses the non-orthogonality of $A$ and $B$.
\end{proof}

\begin{rem}
Notice that, by Theorem~\ref{thm:dim1cohesive}, any o-minimal
expansion~$\CM$ of
a group is cohesive, hence all sets definable in~$\CM$ are cohesive
by~\ref{prop:int-cohesive}. For Theorems~\ref{thm:ominimal}
and~\ref{thm:main2} it is, therefore,
important to work in an arbitrary o-minimal structure.
\end{rem}

\section{Splitting and decomposition}\label{sec:splitting}
In this section we work in an arbitrary structure \cal M. Fix definable orthogonal sets $X_1, \ldots, X_n$. We introduce the notions of splitting and decomposition for definable functions and groups, respectively,  contained in products of the sets $X_i$, as follows. 

\newcommand{\dd}{\textup{Def}(\Pi_i X_i)}

\begin{defn} We let $\dd$ be the collection of all definable sets in $\CM$ that are contained in some cartesian product~$\prod_{j=1}^k X_{t(j)}$ with $t(j) \in
	\{1,\dots, n\}$, for all $j=1,\dots, k$.
\end{defn}



\begin{notation}
	If $S \subseteq \prod_{j=1}^k X_{t(j)}$
	we define
	$$\pi_i: S \to X_i^{k_i}$$ as the projection of $S$ onto the
	$X_i$-components, where $k_i$ is the number of indexes $j$ with $t(j)
	=i$.
	For instance, if $S \subseteq X_1\times X_2 \times X_1$, then $\pi_1:S \to X_1^2$ maps $(a,b,c)$ to $(a,c)$.
\end{notation}

We now introduce the notion of splitting, which will be a crucial tool for our proofs.

	\begin{defn} Let $f:A\to B$ be a function  in $\dd$. We say that $f$ \textbf{splits} (with regard to $X_1, \dots, X_n$)  if for every $x, y\in \dom(f)$ and every $i\leq n$,
$$\pi_i(x)=\pi_i(y) \,\,\,\Rarr\,\,\,\, \pi_i(f(x))=\pi_i (f(y)).$$
That is, $\pi_i (f(x))$ depends only on $\pi_i(x)$. 
Note that if $f$ splits, then up to a permutation of the indexes we can write $f = f_1\times \dots \times f_n |A$ where $f_i: \pi_i(A)\to \pi_i(B)$.
\end{defn}
We will be using the following two facts without specific mentioning.

\begin{fact}\label{fact:splits} Let $f:A\to B$ be a function in $\dd$. Then there is a finite partition ${\mathcal D}$ of $\dom(f)$ into definable sets, such that for each $D\in \mathcal D$, the restriction of $f$ to $D$ splits.
\end{fact}
\begin{proof} By orthogonality of $X_1,\ldots, X_n$, up to a permutation of the variables, $f$ is a finite disjoint union $\bigcup_j f_j$, where $f_j = U_{1,j} \times \dots \times U_{n,j}$ and $U_{i,j} \subseteq \pi_i(A) \times \pi_i(B)$. We conclude observing that each $f_j$ splits.
\end{proof}

Splitting is preserved under composition in the following sense.

\begin{fact}
Let $h, f_1, \dots, f_n$ be  maps in $\dd$, and suppose that the map  $h \circ (f_1\times \dots\times f_n)$ is defined (in the sense that the range of $f_1\times \ldots \times f_n$ is contained in the domain of $h$). If $h,f_1,\ldots, f_n$ split, so does $h \circ (f_1\times \dots\times f_n)$.
\end{fact}
\begin{proof}
Straightforward.
\end{proof}

We now turn to definable groups.

\begin{defn}
	A definable group $G$ \textbf{admits a decomposition} (with respect to $X_1,\ldots, X_n$) if there are definable subsets $A_1,\ldots, A_n$ of $G$, such that $G = A_1\dots A_n$, and $A_i$ is internal to $X_i$, for all $i$.
\end{defn}

Note that, since a set internal to a cohesive set is cohesive (Proposition \ref{prop:int-cohesive}(1)), a definable group admits an orthogonal decomposition (as in the introduction) if and only if there are definable orthogonal cohesive sets $X_1, \dots, X_n$, such that $G$ admits a decomposition with respect to them.

\begin{observation}\label{obs:finite-index}
Let $H$ be a finite index subgroup of~$G$. If $H$ admits a
decomposition with respect to $X_1,\ldots,X_n$, then so does $G$.
\end{observation}

Our goal in the next sections is to prove that if $G$ is definable in an o-minimal structure
and internal to $X_1\times\ldots\times X_n$, then it admits a decomposition in
the   sense above. A relevant case is when $G$ is contained in -- as opposed to being
internal to -- $X_1\times\ldots\times X_n$. In this situation, by
Remark \ref{rem:splits} below, if the group operation of~$G$ splits, then $G$  admits a decomposition. The converse, however, is not true (Example \ref{exa:splits}).
On the other hand, if $G$ is a direct product of groups definable in $X_1,\ldots,X_n$, respectively, then the group operation obviously splits.
In the appendix, we will give an example in which the group operation
splits, yet the group is not even definably isomorphic to a direct
product.

  The rest of this section contains some remarks that help to demonstrate the newly defined notions. 

\begin{rem}\label{rem:splits}
Let $(G,\mu,e)$ be a definable group with $G\subseteq X_1\times\ldots\times X_n$. Then the following are equivalent:
\begin{enumerate}
\item the group operation~$\mu$ splits (with respect to $X_1,\ldots,X_n$);
\item there are definable groups $H_1,\ldots,H_n$, contained in
$X_1,\ldots,X_n$, respectively, such that $G<H_1\times\ldots\times H_n$;
\item there are definable groups $H_1,\ldots,H_n$, contained in
$X_1,\ldots,X_n$, respectively, such that $G$ is a finite index subgroup
of~$H_1\times\ldots\times H_n$;
\end{enumerate}
and, if either of these holds, then $G$ admits a decomposition.
\end{rem}
\begin{proof}
Assume $(1)$. Consider the projections~$H_1,\ldots,H_n$ of~$G$
on~$X_1,\ldots,X_n$, respectively. The group operation~$\mu$, by the
splitting hypothesis, takes the form
\[\mu((x_1,\ldots,x_n),(y_1,\ldots,y_n)) =
(\mu_1(x_1,y_1),\ldots,\mu_n(x_n,y_n))\]
It is straightforward to check that $\mu_i$ is a group operation on~$H_i$, and
$(2)$ is thus established.

Now assume (2). Let $\pi_i: \prod_i X_i \to X_i$ be the projection onto the $i$-th component. Replacing $H_i$ with $\pi_i(G)$ we can assume, without loss of generality, that $\pi_i(G) = H_i$ for each $i$.
We prove that $G$ has finite index in~$H_1\times\ldots\times H_n$.
Let $$p_i:\prod_j X_j \to \prod_{j\ne i} X_j$$ be the projection that
omits the $i$-th coordinate.
Fix an index~$i$ and consider an element $k$ of~$\prod_{j\neq i} H_j$. Let
$H_i(k) = \pi_i(G\cap p_i^{-1}(k))$.
Observe that $L_i:=H_i(p_i(e))$ is a subgroup of~$H_i$. We claim that
$L:=L_1\times\ldots\times L_n$ has finite index in~$H_1\times\ldots\times
H_n$. Assuming the claim, since $L<G$, also $G$ must have finite index.
It suffices to prove that $L_i$ has finite index in~$H_i$. The coset
$h L_i$, for $h\in H_i$, coincides with~$H_i(p_i(g))$ for any
$g\in G\cap\pi_i^{-1}(h)$. Thus, in particular, the cosets of~$L_i$ belong
to the family~$H_i(-)$ indexed over~$\prod_{j\ne i} X_j$. This family is
finite by Fact~\ref{fact:family}, and this concludes the proof of~$(3)$.

It is immediate that $(2)$, hence also $(3)$, implies~$(1)$. It remains to
show that $G$ admits a decomposition. To this aim, observe that $L$ is
decomposable: it is, in fact, the product of the subgroups~$G\cap
p_i^{-1}(p_i(e)) \cong L_i$. We conclude by Observation~\ref{obs:finite-index}.
\end{proof}

\begin{example}\label{exa:splits}
We give an example of a group whose operation does not split, even up to
definable isomorphism, but the group admits a decomposition. Let $\mathcal M = \mathcal R_1 \sqcup \mathcal R_2$ where $\mathcal R_1$ and $\mathcal R_2$ are isomorphic copies of $(\R, <, +)$ and note that their domains $R_1$ and $R_2$ are orthogonal in $\mathcal M$. 
Let $(R_1\times R_2, +,0)$ be the product group. 
Fix $a\in [0,1)$. Consider the lattice $\Lambda= \Z\cdot(0, 1)
+ \Z\cdot(1, a)\subseteq R_1 \times R_2$ and the definable set $[0,1)^2 \subseteq R_1\times R_2$. Define $G=([0,1)^2, \mu, 0)$, where $\mu = + \mod
\Lambda$. Clearly, $G=A_1+A_2$, where $A_1=[0, 1)\times \{0\}$ and
$A_2=\{0\}\times [0,1)$, and hence it admits a decomposition. It is easy to see that if $a\neq0$, then $\mu$ does
not split. Moreover, $a\in \Q$, if and only if $G$ is definably isomorphic
to a group whose operation splits, if and only if $G$ is definably
isomorphic to a direct product of groups definable in~$R_1$ and~$R_2$,
respectively.
\end{example}

\section{Compact domination}\label{sec:cdom}Let again $\CM$ be an arbitrary structure and  $G$ a definable group.
We call a definable set $S\sub G$ {\em left-generic} ({\em right-generic}) if finitely many left-translates (respectively right-translates) of $S$ cover $G$. We call $S$ {\em generic} if it both left-generic and right-generic (which are equivalent when $G$ has {\em finitely satisfiable generics}
({\em fsg})~\cite[Proposition~4.2]{Hrushovski2007a}).
We are mainly interested in o-minimal structures, but in this section we consider the larger NIP class.  We recall that NIP structures include both the o-minimal structures (e.g. the real field) and the stable structures (e.g. the complex field). If $\CM$ is a NIP structure and $G$ is {\em compactly dominated} (see \cite{Simon2015} for the definitions), then $G$ has {\em fsg}
\cite[Theorem~8.37]{Simon2015}, \cite[Proposition~3.23]{Starchenko2016}, \cite[Proposition~8.33]{Simon2015}. Compact domination is a model-theoretic form of compactness: for instance a semialgebraic linear group $G<GL(n,\R)$ is compactly dominated if and only if it is compact. The following proposition subsumes all we need about these notions.
\begin{prop}\label{prop:CDgenerics}
	Let $\CM$ be a NIP structure and  $G$  a compactly dominated group definable in $\CM$. Then:
	\begin{enumerate}
		\item If the union of two definable subsets of $G$ is generic, then one of the two is generic.
		\item Suppose $G = HK$, where $H$ and $K$ are definable subgroups and $K$ is normal. Then $S\subseteq G$ is generic if and only if it contains  a set of the form $AB$ where $A$ is a generic subset of $H$ and $B$ is a generic subset of $K$. \footnote{Below we only need the case when $G$ is the direct product of $H$ and $K$.}
		\item If $G$ and $H$ are compactly dominated, then $G\times H$ is compactly dominated.
	\end{enumerate}
\end{prop}
\begin{pf} Point (1) holds for all groups with
{\it fsg}~\cite[Proposition~4.2]{Hrushovski2007a}.

	To prove point (2) we may assume $\CM$ is $\kappa$-saturated for some sufficiently big cardinal $\kappa$. We make use of the {\em infinitesimal subgroup} $G^{00}$ (see \cite{Simon2015} for the definition). Specifically, we need to observe that in a compactly dominated group $G$, a subset $S$ is generic if and only some translate of $S$ contains $G^{00}$ \cite[Proposition 2.1]{Berarducci2009a}. If $G = HK$ with $H<G$ and $K\lhd G$, then $G^{00} = (HK)^{00} = H^{00}K^{00}$ \cite[Theorem 4.2.5]{CONVERSANO2009} (this holds without the hypothesis that $G$ is compactly dominated). Now suppose $S\subseteq G$ is generic. So there is $g\in G$ such that $gG^{00} \subseteq S$. Now, $K^{00}$ is a normal subgroup of $G$ (being a definably characteristic subgroup of $K$), and so, writing $g=kh$ for $h\in H$ and $k\in K$, we get $gG^{00} = kh K^{00}H^{00} =  (kK^{00})(hH^{00}) \subseteq S$.
	By $\kappa$-saturation there are definable sets $U,V$ with $K^{00}\subseteq U \subseteq K$ and $H^{00}\subseteq V \subseteq H$ such that $(kU)(hV) \subseteq S$.
	By construction, $kU$ and $hV$ are generic in $K$ and $H$, respectively.
	
	
	Point (3) follows from \cite[Corollary 3.17]{Starchenko2016} the product of smooth measures is smooth) and the fact that a group is compactly dominated if and only if it has a smooth left-invariant measure \cite[Theorem 8.37]{Simon2015}.
\end{pf}

\begin{fact}
 Definably compact groups in an o-minimal structure are compactly dominated.
\end{fact}
\begin{pf} This was first proved for o-minimal expansions of a field in \cite{Hrushovski2011a}. We give some bibliographical pointers to obtain the result for arbitrary o-minimal structures. First one shows that a definably compact group in an o-minimal structures has {\it fsg} \cite[Theorem 8.6]{Peterzil2007}. From this one deduces that $G$  admits a {\em generically stable left-invariant measure} \cite[Proposition 8.32]{Simon2015}. In an o-minimal structure (and more generally in a {\em distal} structure), a generically stable measure is {\em smooth} \cite[Proposition 9.26]{Simon2015}. 
Finally, a NIP group with a smooth left-invariant measure is compactly dominated \cite[Theorem 8.37]{Simon2015}.
\end{pf}

\begin{prop}\label{prop:fsgsplits} Let $G$ be a compactly dominated group. 	Given two generic sets $A\subseteq G$ and $B\subseteq G$, there is $h\in G$ such that $A\cap hB$ is generic.
\end{prop}
\begin{pf}
	There is a finite subset $I\subseteq G$ such that $A \subseteq \bigcup_{h\in I} hB$. By Proposition \ref{prop:CDgenerics}(1), there is $h\in I$ such that $A \cap hB$ is generic.
\end{pf}

\section{Decomposition of compactly dominated abelian groups
(NIP)}\label{sec:NIP}
In this section $\CM$ is a NIP structure and $X_1,\dots, X_n$ are orthogonal definable sets.

\begin{thm} \label{thm:NIP} Let $G$ be a compactly dominated abelian group contained in $X_1 \times \dots \times X_n$. Then $G$ admits a decomposition with respect to $X_1, \dots, X_n$.
\end{thm}
\begin{proof}
	Let $P_n: G^n\to G$ be the function sending $(x_1,\ldots,x_n)$ to $\prod_{i=1}^{n}x_i$. Then $P_n\in \dd$. Since $G$ is compactly dominated, so is $G^n$ (Proposition \ref{prop:CDgenerics}(3)).  By Fact \ref{fact:splits} and Proposition \ref{prop:CDgenerics}(1), $P_n$ splits on a generic definable set $S\subseteq G^n$. By \prettyref{prop:CDgenerics}(2) and induction on $n$ we find definable generic sets $A_1, \ldots, A_n \subseteq G$ such that $A_1\times \ldots \times A_n \subseteq S$.
	By Proposition \ref{prop:fsgsplits} (and induction on $n$) we find $a_1, \ldots, a_n \in G$ such that the set $$U = \bigcap_{i=1}^n a_i A_i$$ is generic in $G$.
	Again by Fact \ref{fact:splits}  and induction, there is a generic set $D\subseteq U$ such that for every $i=1,\ldots, n$ the function $$f_i: x \in D \mapsto a_i^{-1}x$$ splits on $D$. Since $D\subseteq a_i A_i$, the image of $f_i$ is contained in $A_i$. It follows that the function
	$$f_1 \times \ldots \times f_n: D^n \to A_1\times \ldots \times A_n$$
	splits. Since $P_n$ splits on $A_1\times \ldots \times A_n$, we deduce that the function
	$$f= P_n \circ (f_1\times \ldots \times f_n)$$
	splits on $D^n$.
	Since $G$ is abelian,
	\begin{equation}\label{eq:f}
	f(x_1, \ldots, x_n) = a \prod_{i=1}^{n}x_i
	\end{equation}
	where $a = \prod_{i=1}^n a_i^{-1}$.
	By the orthogonality assumption, $D$ is a finite union of sets of the form $U_1\times \ldots \times U_n$ with $U_i \subseteq X_i$. By Proposition \ref{prop:CDgenerics}(1) one of these sets is generic, so by replacing $D$ with a smaller set we can assume that
	$$D = U_1\times \ldots \times U_n.$$
	Let $\pi_i: \prod_i X_i\to X_i$ be the projection onto the $i$-th component and let $$p_i:\prod_j X_j \to \prod_{j\ne i} X_j$$ be the projection that omits the $i$-th coordinate.
	Now fix $k\in D$ and let
	$$D_i = p_i^{-1}p_i(k) \cap D = \{x\in D \mid \forall j\neq i \; \pi_j(x)= \pi_j(k) \}.$$  Notice that $D_i$ is $U_i$-internal, hence {\em a fortiori} it is $X_i$-internal.
	We claim that
	 \begin{equation}\label{eq:splitD}
	 k^{n-1}D \subseteq D_1\dots D_n
	 \end{equation}
	To prove the claim, let $g\in D$ and let $x_i\in G$ be such that
	$$\pi_i(x_i) = \pi_i(g)\; \& \; \pi_j(x_i) = \pi_j(k)  \; \text{for} \; j\neq i.$$ Notice that $x_i\in D_i$.
Since $f$ splits on $D^n$, the value of $\pi_i f(x_1,\ldots, x_n)$ does not change if we replace $x_i$ with $g$ and $x_j$ with $k$ for $j\neq i$. By Equation (\ref{eq:f}) we then obtain $$\pi_i f(x_1,\dots, x_n) = \pi_i(ak^{n-1}g).$$  Since this holds for every $i$, we deduce that $$f(x_1,\dots, x_n) = ak^{n-1}g$$ and since $f(x_1,\ldots, x_n) = a \prod_{i=1}^n x_i$ we obtain Equation \ref{eq:splitD}.

We have thus shown that a translate of $D$ is contained in a product of $X_i$-internal sets. Since $D$ is generic, the same holds for $G$, so $G$ is a product of $X_i$-internal sets.
\end{proof}

\section{Decomposition of definably compact abelian groups (o-minimal)}\label{sec:omin1}

In this section $\CM$ is an o-minimal structure and $X_1,\dots, X_n$ are orthogonal definable sets. We prove the following variant of Theorem \ref{thm:NIP}, where the NIP hypothesis is replaced by o-minimality, but the group is only assumed to be internal
to $X_1\times \ldots \times X_n$. 

\begin{thm} \label{thm:compact} Let $G$ be a definably compact abelian group internal to $X_1 \times \dots \times X_n$. Then $G$ admits a decomposition with respect to $X_1, \dots, X_n$.
\end{thm}

We need the following lemma.
\begin{lem}\label{lem:embed}
	Let $X$ be an infinite set definable in an o-minimal structure $\CM$. Then there is a definable set $Y = \omin{X}$ such that:
	\begin{enumerate}
		\item
		$X$ and $Y$ are internal to each other;
		\item there is $k \in \N$ and a definable injective map from $X$ to $Y^k$;
		\item $Y$ has a definable linear order $\prec$ such that $(Y,\prec)$ with the induced structure from $\CM$ is o-minimal.\footnote{The induced structure contains a predicate for each $\CM$-definable subset of $Y^n$ for $n\in \N$.} We call the resulting structure $\CY$ the {\bf o-minimal envelope} of $X$.
	\end{enumerate}
\end{lem}
\begin{pf}
	Suppose $X\subseteq M^m$ and let $\pi_i:M^m \to M$ be the projection onto the $i$-th coordinate ($i = 1,\dots,m$). Fix parameters $a_1 < \dots < a_m$ in $M$ and let $Z = \bigcup_i \{a_i\}\times \pi_i(X) \subseteq M^2 $. Then $Z$ has dimension $1$ and is bi-internal to $X$. Each $\pi_i(X)$ is a finite union of open intervals and points with induced order from $\CM$. We order $Z$ lexicographically, i.e. all the elements of $\{a_i\}\times \pi_i(X)$ preceed all the elements of $\{a_{i+1}\} \times \pi_{i+1}(X)$. Adding and removing from $Z$ finitely many points we obtain a set $Y$ satisfying point (3) (we need each pair of consecutive open intervals to be separated by exactly one point). Points (1) and (2) are clear from the construction.
\end{pf}

\begin{defn}[{\cite[Def. 1.1]{Ramakrishnan2014}}]\label{defn:quotients}
	Let $X,Y$ be definable sets, $E_1,E_2$ two definable equivalence relations on $X$ and $Y$ respectively. A function $f : X/E_1 \to Y/E_2$ is called definable if the set $\{(x,y) \in X\times Y \mid f([x]_{E_1}) = [y]_{E_2}\}$ is definable.
\end{defn}

\begin{prop}\label{prop:intdef}
	Let $G$ be a definable group in an o-minimal structure $\CM$ and let $X_1, \dots, X_n$ be definable sets in $\CM$.  If $G$ is internal to $X_1\times \dots \times X_n$, then there is an injective definable map $f:G\to {X'_1}\times  \dots \times {X'_n}$, where each $X'_i$ is bi-internal to $X_i$ ($i=1,\dots,n$).
\end{prop}
\begin{pf}
	By 	Lemma \ref{lem:embed} $G$ is internal to $Y= \omin{X_1\times \dots \times X_n}$, so it can be considered as an interpretable group in the o-minimal structure $\mathcal Y$. By \cite{Ramakrishnan2014} interpretable groups in an o-minimal structure are definably (in the sense of Definition \ref{defn:quotients}) isomorphic to definable groups. It follows that there is a definable injective map $f:G \to Y^k$ for some $k\in \N$. By the construction of the o-minimal envelope, $Y$ can be embedded into a product of sets $Y_i$, where $Y_i$ is bi-internal to $X_i$ (if $X_i \subseteq M^n$, it suffices to define $Y_i$ as $\pi_1(X_i) \times \dots \times \pi_n(X_i)$). Now it suffices to take $X'_i = Y_i^k$.
\end{pf}

We are now ready to finish the proof of the theorem.

\begin{proof}[Proof of Theorem \ref{thm:compact}]
By  Proposition \ref{prop:intdef} there are definable sets $X_1',\dots, X_n'$ with $X_i'$ bi-internal to $X_i$ such that $G$ is definably isomorphic to a group $G'$ contained in $X_1'\times \dots \times X_n'$. Since $G$ is definably compact, $G'$ also is. By Theorem \ref{thm:NIP} $G'$ admits a decomposition with respect to $X_1', \dots, X_n'$, hence also with respect to $X_1,\dots, X_n$. Hence so does $G$.
\end{proof}

\section{Decomposition: general case (o-minimal)}\label{sec:omin2} In this section, we prove our main theorems (\ref{thm:ominimal} and \ref{thm:main2}). Fix an o-minimal structure \cal M and definable orthogonal sets $X_1, \ldots, X_n$.
We first prove that the existence of decompositions is preserved under taking central extensions.

\begin{lem}\label{lem:central}
	Let $1\to N \to G \stackrel{f}{\to}  H \to 1$ be a definable exact
sequence of definable groups internal to $X_1\times \dots \times X_n$ with
$N<Z(G)$.  If $N$ and $H$ admit a decomposition with respect
to~$X_1,\ldots,X_n$, then $G$ too admits a decomposition with respect to the same
orthogonal sets.
\end{lem}
\begin{pf}
	By  assumption, we can write $H = H_1 \dots H_n$ and $N = N_1 \dots N_n$, where $H_i$ and $N_i$ are $X_i$-internal definable sets (not necessarily subgroups). We have
$$G = f^{-1}(H_1) \dots f^{-1}(H_n).$$
 By \cite[Theorem 2.5]{EDMUNDO2003}, there is a definable section $\sigma: H\to G$. Since $f^{-1}(H_i) = \sigma(H_i) N$, we have
 $$G = \sigma(H_1)N \dots \sigma(H_n)N.$$ Since $N= N_1\dots N_n$ is contained in the center of $G$,  it follows that $G = U_1\dots U_n$, where $U_i$ is the $X_i$-internal set $\sigma(H_i) N_i\dots N_i$ ($n$ occurrences of $N_i$).
\end{pf}

We can now handle the abelian case.

\begin{prop}\label{prop:abelian} Let $G$ be a definable group internal to $X_1\times \dots \times X_n$.
	If $G$ is abelian, then $G$ admits a decomposition with respect
to~$X_1,\ldots,X_n$.
\end{prop}
\begin{pf}   We reason by induction on dimension. For $\dim (G)=1$, $G$ is
cohesive by  \prettyref{thm:dim1cohesive}, so it is indecomposable (Proposition \ref{prop:int-cohesive}), hence it is internal to one of the $X_i$.
Let $\dim (G)>1$.
If $G$ is definably compact, then we conclude by~\ref{thm:compact}.
If not, then by \cite[Theorem 1.2]{Peterzil1999}, $G$ has a
$1$-dimensional torsion-free definable subgroup~$H<G$. Since $\dim(H)=1$,
$H$ admits a decomposition. By induction on dimension, so does $G/H$.
Therefore, since $G$ is
abelian, we can apply \prettyref{lem:central}.
\end{pf}

The following fact must be well-known, but we include a proof for completeness.

\begin{fact}\label{fact:GmodZ}
	Let $G$ be a connected group definable in an o-minimal structure. If $Z(G)$ is finite, then $G/Z(G)$ is centerless.
\end{fact}
\begin{pf} Let $a\in G$, such that $aZ(G)$ is in the center of $G/Z(G)$. We want to prove that $a\in Z(G)$. We have that for all $b\in G$, $a^{-1}b^{-1}ab\in Z(G)$. Since $G$ is connected, the image of the map $f:G\to Z(G)$ sending $b\in G$ to $a^{-1}b^{-1}ab \in Z(G)$ is connected. Since $Z(G)$ is finite, $f$ must be constant. Since $f$ maps the identity $e\in G$ to $e$, we have $a^{-1}b^{-1}ab = e$. Thus $a\in Z(G)$, as needed.
\end{pf}

We can now prove our first main result.

\begin{thm}\label{thm:ominimal} Let $G$ be internal to
$X_1\times \dots \times X_n$ where $X_1,\ldots, X_n$ are
 orthogonal definable sets. Then $G$  admits a
decomposition with respect to~$X_1,\ldots,X_n$.
\end{thm}
\begin{proof}
We observe that, since the connected component of the identity~$G^0$ has
finite index
in~$G$, by Observation~\ref{obs:finite-index}, it suffices to find a
decomposition of~$G^0$. We may thus assume that $G$ is
connected.

We prove, now, the theorem, by induction on $\dim(G)$. For $\dim(G)=0$, it is obvious. Assume $\dim (G)>0$. \smallskip

  	Case 1. Suppose that the centre~$Z(G)$ is finite.
	Then $Z(G)$ admits a decomposition, and by \prettyref{lem:central} it
suffices to prove that $H=G/Z(G)$ has a decomposition.
Since $Z(G)$ is finite, $H$ is centerless (\prettyref{fact:GmodZ}).

By~\cite{Ramakrishnan2014}, $H$ is definably isomorphic to a definable
group. By \cite[Theorems 3.1 and 3.2]{Peterzil2000},
	it follows that there are definable real closed fields $R_1,
\dots, R_k$ and definable linear groups $H_i<GL(n,R_i)$ such that $H$ is
definably isomorphic to $H_1\times \dots \times H_k$.
A definable real closed field in an o-minimal structure has dimension $1$
\cite[Theorem 4.1]{Peterzil1999}. By Theorem~\ref{thm:dim1cohesive} each $R_i$ is cohesive, hence internal to some $X_j$ by Proposition \ref{prop:int-cohesive}(2). It follows that the each of the 
subgroups~$H_1,\ldots,H_k$ is internal to one
of the $X_j$. Therefore $H$ admits a decomposition with respect to the orthogonal sets~$X_1,\ldots,X_n$. \smallskip

	  Case 2. Suppose $Z(G)$ is infinite. By the abelian case
(\prettyref{prop:abelian}), $Z(G)$ has a decomposition with respect
to~$X_1,\ldots,X_n$. By induction on the dimension, $G/Z(G)$ has a
decomposition too. Therefore we can conclude by~\prettyref{lem:central}.
\end{proof}	
As a by-product of the proof we obtain:
\begin{prop}\label{prop:definably-simple}
	If $G$ is definably simple, then $G$ is cohesive.
\end{prop}
\begin{pf}
	Let $G$ be definably simple. By the proof of Theorem \ref{thm:ominimal} there is a definable real closed field $R$ such that $G$ is definably isomorphic to a definable subgroup of $GL(n,R)$. Since $\dim(R) = 1$, by Theorem \ref{thm:dim1cohesive} $R$ is cohesive. But $GL(n,R)$ is internal to $R$, thus all its definable subsets are cohesive.
\end{pf}

We now proceed towards our second main result.

\begin{lem}\label{lem:orthdec}
	Let $G$ be an interpretable group. Then there are cohesive orthogonal definable sets $X_1, \dots, X_n$ and an injective map $h:G\to \prod_{i=1}^n X_i$.
\end{lem}
\begin{pf}
By \cite[Theorem 3]{Ramakrishnan2014}, there is a definable injective map
$f:G \to \prod_{j=1}^k G_j$ where each  $G_j$ is a $1$-dimensional
definable group.  Define a relation $R$ on $\{1,\dots, k\}$ by $iRj$ if
$G_i$ and $G_j$ are not orthogonal.  By Theorem~\ref{thm:dim1cohesive}, the groups~$G_i$ are cohesive, hence $R$
is an equivalence relation.  Suppose there are $n$ equivalence classes.
Let $X_i$ be the product of the groups~$G_j$, with $j$ in the
$i$-th class.
Using $f$, it is easy to define (by a permutation of the coordinates on
the image) the injective map $h:G\to \prod_{i=1}^n X_i$.

The sets~$X_i$ are cohesive by Proposition~\ref{prop:prod-cohesive} and mutually
orthogonal by Corollary~\ref{cor:both}.
\end{pf}

\begin{thm}\label{thm:main2} If $G$ is a group interpretable in an
o-minimal structure $\CM$, then $G$ admits an orthogonal decomposition
$G=A_1 \dots A_n$.
\end{thm}
\begin{proof} Let $X_1, \dots, X_n$ be the sets provided by
Lemma~\ref{lem:orthdec}. By Theorem~\ref{thm:ominimal}, there are $X_i$-internal sets $A_i\sub G$, $i=1,\dots, n$, such that
$$G=A_1 \dots A_n.$$ Clearly, the $A_i$'s are orthogonal and cohesive, as they inherit those properties from the $X_i$'s.
\end{proof}

\begin{cor}
If $G$ is infinite, in Theorem \ref{thm:main2} we can choose each $A_i$ to be infinite. In this case, the number $n$ is an invariant of $G$ up to definable isomorphism.  Indeed, if $G=B_1\ldots B_m$ is another decomposition of $G$ as a product of orthogonal cohesive infinite definable subsets, then $m=n$ and each $B_i$ is bi-internal to a unique $A_j$.
\end{cor}
\begin{pf}  Suppose that $G$ is infinite and fix an orthogonal decomposition $G = A_1\ldots A_n$. Then at least one $A_i$ is infinite, say $A_1$. If some $A_i$ is finite we may replace $A_1$ with $A_1A_i$ and omit $A_i$ obtaining another valid orthogonal decomposition. So we may assume that $A_1, \ldots, A_n$ are all infinite. Now consider another decomposition $G = B_1\ldots B_m$ into infinite orthogonal cohesive sets. Fix $B_i$ and observe
that $B_i$ is internal to $G$, which is internal to the cartesian product
$A_1\times \ldots \times A_n$. Since $B_i$ is indecomposable, it must be
internal to some $A_j$. Moreover, $j$ must be unique, because if $B_i$ is internal to both $A_j$ and
$A_h$, with $j\ne h$, then it is non-orthogonal to both, so by cohesiveness of $B_i$,
$A_i$ and $A_h$ are non-orthogonal, a contradiction. The argument also shows that $m=n$ and $B_i$ is in fact bi-internal to the corresponding $A_j$.
\end{pf}

\section{Locally definable groups}
In this section, we fix again an o-minimal structure \cal M, and prove Theorem \ref{thm:lattice2}. Let us first recall a few definitions concerning locally definable sets and groups (which can, in fact, be given for  arbitrary structures).

\begin{defn}
	A {\bf locally definable set} $X$ is a countable union of definable sets together with a given presentation as such a countable union. A subset of a locally definable set $X$ is said to be {\bf definable} if it is definable in $\CM$ and is contained in the union of finitely many sets of the presentation of $X$ (this last condition is automatically satisfied if $\CM$ is $\aleph_0$-saturated). A {\bf compatible} subset of a locally definable set $X$ is a subset which intersects every definable subset of $X$ at a definable set. A compatible subset is {\bf discrete} if it intersects every definable set into a finite set. A {\bf locally definable function} is a function between locally definable sets whose restriction to each definable set is definable. Similar definitions apply to groups. A locally definable group is a locally definable set with a locally definable group operation. We can then speak of compatible and discrete subgroups and locally definable homomorphisms. A locally definable group is {\bf definably generated} if it is generated by a definable subset. 	
\end{defn}

\begin{thm} \label{thm:lattice2}
	Let $G$ be an abelian group definable in the disjoint union $\bigsqcup_i {\mathcal X}_i$ of finitely many o-minimal structures ${\mathcal X}_1, \ldots, {\mathcal X}_n$. Then there is a locally definable isomorphism $$G \cong G_1\times \ldots \times G_n /\Gamma,$$ where $G_i$ is a locally definable and definably generated group in ${\mathcal X}_i$, and $\Gamma$ is a compatible locally definable discrete subgroup of $G_1\times \ldots \times G_n$.
\end{thm}
\begin{pf} The structure $\bigsqcup_i {\mathcal X}_i$ is bi-interpretable with the o-minimal structure $\CM$ obtained by concatenating ${\mathcal X}_1, \ldots, {\mathcal X}_n$ in the given order and adding $n-1$ points to separate ${\mathcal X}_i$ from ${\mathcal X}_{i+1}$ for $i<n$. We can therefore apply to $\CM = \bigsqcup_i {\mathcal X}_i$ the various results concerning o-minimal structures.
By Theorem \ref{thm:ominimal}, the group $G$ admits a decomposition $G = A_1\ldots A_n$ with respect to $X_1,\ldots, X_n$, where $X_i$ is the domain of ${\mathcal X}_i$.
Let $\langle A_i \rangle$ be the locally definable subgroup of $G$ generated by $A_i$.

We claim that $\langle A_i \rangle$ is locally definably isomorphic to a definably generated group $G_i$ in the structure ${\mathcal X}_i$. To this aim, let $A_i^{(n)} \subseteq G$ consist of the $n$-fold products $a_1\ldots a_n$, where each $a_i$ is either an element of $A_i$ or is the group-inverse of an element of $A_i$. Without loss of generality, after permuting coordinates, we can choose $k_1,\ldots, k_n\in \N$ such that $G$ is included in $X_1^{k_1}\times \ldots \times X_n^{k_n}$.
	Since $A_i^{(n)}$ is $X_i$-internal and included in $G$, it must have a finite projection on the factors different from $X_i$. Thus we can write $A_i^{(n)} \subseteq L_n \times X_i^{k_i} \times F_n$ where $L_n$ and $F_n$ are finite sets. It follows that the subgroup $\langle A_i \rangle$ of $G$ generated by $A_i$ is included in $L\times X_i^{k_i}\times F$ where $L = \bigcup_n L_n$ and $F= \bigcup_n F_n$ are countable sets. Consider a bijection sending $L\cup F$ to a countable subset of $X_i$. This induces a locally definable bijection between $\langle A_i \rangle$ and a locally definable subset $G_i \subseteq X_i^{k_i+1}$. We can endow $G_i$ with a group operation via the bijection. The resulting group $G_i$ will then be locally definable, and in fact definably generated, in the structure $\cal X_i$. There is a locally definable group homomorphism $f: G_1\times \ldots \times G_n \to G$ induced by the composition $G_1 \times \ldots \times G_n \cong \langle A_1 \rangle \times \ldots \times \langle A_n \rangle \to G$, where the last map sends $(x_1,\ldots, x_n)$ to their product in $G$. Note that $f$ is a homomorphism since $G$ is abelian.

It remains to prove that the kernel $\Gamma$ of the above $f$ is  discrete. To this aim fix, for $i=1,\ldots, n$, a definable subset $U_i$ of $\langle A_i \rangle$ and let $S$ be the set of all tuples $(a_1,\ldots, a_n)\in U_1\times \ldots \times U_n$ such that $a_1 \ldots a_n = 1_G$. It suffices to show that $S$ is finite.  The sets $U_1,\ldots, U_n$ are orthogonal by Proposition \ref{prop:intort}, since they are internal to $A_1,\ldots, A_n$ respectively. It follows that $S$ is a finite union of sets of the form $B_1\times \dots \times B_n$ with $B_i\subseteq U_i$. However, each $B_i$ can only be a singleton because any choice of $n-1$ elements from $a_1,\ldots, a_n$ determines the last one via the equation $a_1\ldots a_n = 1_G$.
\end{pf}

\section{Questions}\label{sec:questions}
	\begin{enumerate}
		\item Does the conclusion of Theorem \ref{thm:ominimal} extend to the case when $\CM$ is an arbitrary structure?
		\item Given a structure $\CM$ and a definably simple group $G$ in $\CM$, is $G$ always cohesive? (The answer is positive if $\CM$ is o-minimal, by Proposition \ref{prop:definably-simple} One could then consider the case when $\CM$ is merely assumed to be NIP.)
		\item Can the abelianity hypothesis in Theorem~\ref{thm:lattice2} be removed?
		\item Are groups of dimension $1$ in a geometric theory always cohesive? (The answer is positive in the o-minimal context, by Theorem \ref{thm:dim1cohesive}.)
\item Is a set internal to an indecomposable set also indecomposable?
 \item Does Proposition \ref{prop:XY-box} hold without the saturation hypothesis?
 \item Is it true that if a definable group $G$ admits a decomposition with regard to orthogonal definable sets $X_1, \dots, X_n$, then so does every definable subgroup of $G$?
	\end{enumerate}

\section{Appendix}
We construct an example of a definable group whose group
operation splits, but the group is not a product of two infinite
groups, hence in particular two orthogonal groups.
First we need the following observation.

\begin{exa}\label{exa:heise}
	There is a real Lie group $G$, definable in the pure real field
structure, which is not a semidirect product of the connected component of the identity $G^0$ and a finite non-trivial group.
\end{exa}
\begin{pf}
	Let $p$ be an odd prime. The Heisenberg group mod $p$ is the semialgebraic group $H$ of matrices of the form
	$\begin{bmatrix}
	1 & a & c \\
	0 & 1 & b \\
	0 & 0 & 1
	\end{bmatrix}$ where $c \in \R/p\Z$ and $a,b\in \Z/p\Z$. We claim
that $H$ is not Lie isomorphic to a semidirect product of a connected real
Lie group and a discrete group. Taking $a,b=0$ we obtain the center $Z(H)$
of $H$, which coincides with $H^0$ and it is isomorphic to the circle
group $\R/p\Z$. The quotient $H/H^0$ is isomorphic to $(\Z/p\Z)^2$.
	Since $H$ is not abelian, $H$ is not isomorphic to the direct
product  $H^0 \times (\Z/p\Z)^2$. Moreover the direct product is the only
possible semidirect product in the Lie category because $(\Z/p\Z)^2$ has
no non-trivial continuous action on $\R/p\Z$ (since the only non-trivial
definable automorphism of $\R/p\Z$ is the inverse, and $p$ is odd).
\end{pf}

\begin{exa}\label{exa:DH}
	Let $R_1$ and $R_2$ be two orthogonal copies of the field $\R$ and
work in the o-minimal structure $M = R_1\sqcup R_2$ obtained by
concatenation of $R_1$ and $R_2$ with a separating element between them.
Let $H_i$ be the Heisenberg group mod~$3$ over~$R_i$.
Consider the definable
group $H_1\times H_2$. Now consider the definable subgroup $G < H_1 \times
H_2$ consising of the pairs of matrices
	$$\left\langle \begin{bmatrix}
	1 & a & c \\
	0 & 1 & b \\
	0 & 0 & 1
	\end{bmatrix}, \begin{bmatrix}
	1 & a & c' \\
	0 & 1 & b \\
	0 & 0 & 1
	\end{bmatrix} \right\rangle$$
	with $a,b\in \Z/3\Z, c,c' \in \R/3\Z$. Note that $G$ is definably
isomophic to an extension of $(\R/3\Z)^2$ by $(\Z/3\Z)^2$. In Proposition \ref{prop:DH} below we prove that
$G$ is not a direct product of two infinite definable subgroups. This
shows that, although the group operation splits with respect to~$R_1$
and~$R_2$ (because it is induced by
the direct product $H_1\times H_2$), $G$ is not a direct product of
orthogonal subgroups.
\end{exa}

\begin{prop}\label{prop:DH}
	The group $G$ in Example \ref{exa:DH} is not a direct product of two infinite definable subgroups.
\end{prop}

\begin{pf} Assume that $G$ is the direct product of two definable infinite
subgroups $G_1$ and $G_2$. Then $\dim(G_1) = \dim(G_2) = 1$. We may assume
that $G_1$ is $R_1$-internal and $G_2$ is $R_2$-internal.
	Consider the natural (surjective) projections $\pi_i: G  \to H_i$.
	
	We claim that $G_1^0=\pi_2^{-1}({\e})$ and $G_2^0 = \pi_1^{-1}({\e})$
where $G_i^0$ is the connected component of the identity of $G_i$.
Consider for instance $G_1$. Clearly $\pi_2(G_1)$ is finite by
orthogonality. Hence $\pi_2^{-1}({\e}) \cap G_1$ has finite index in $G_1$,
so it is infinite. Moreover $\pi_2^{-1}({\e})$ is $H_1^0\times\{{\e}\}$, hence
it is connected and, since two definably connected one dimensional groups
having infinite intersection coincide, it must coincide with $G^0_1$. The claim
is thus proved.
	
	Observe that $Z(G) = G^0 = G^0_1 \times G^0_2$,
thus $[G_1:G^0_1][G_2:G^0_2]= [G:G^0] =
9$. So, there are three cases for the possible values of the indexes of
$G^0_1$ and $G^0_2$: $(1,9), (3,3), (9,1)$.

\medskip
First case: $[G_1:G_1^0]=1$ and $[G_2:G_2^0]=9$ (observe that the third
case is symmetric). In this case $G_1$ is connected, thus
$G_1=G_1^0$,
and $|\pi_1(G_2)| = 9$ because $G^0_2 = \pi^{-1}(\e)$ has index $9$ in $G_2$.
On the other hand
$$H_1 = \pi_1(G)  = \pi_1(G_1)\pi_1(G_2) = H_1^0\pi_1(G_2) = Z(H_1)\pi_1(G_2)$$
and since $\pi_1(G_2)$ has $9$ elements and $Z(H_1)$ has index $9$ in $H_1$, it follows that $H_1$ must be the direct product of $Z(H_1)$ and $\pi_1(G_2)$,
contradicting the claim in \prettyref{exa:heise}.

\medskip
Second case: $[G_1:G_1^0]=3 = [G_2:G^0_2]$. In this case we show that $G_1$ and $G_2$ are abelian and we reach a contradiction since $G$ is not abelian. By symmetry it suffices to show that $G_1$ is abelian. First recall that $G_1^0<Z(G)$, so in particular $G_1^0$ is central in $G_1$, and definably isomorphic to $\R/3\Z$. It follows that $G_1$ is definably isomorphic to a central  extension of $\R/3\Z$ by $\Z/3\Z$. We claim that such a group is necessarily abelian. To this aim we show that there is a copy of $\Z/3\Z$ which is a complement of $G_1^0$.  Let $G_1^0, aG_1^0, bG_1^0$ be the three connected components of $G_1$. Note that the map $x\mapsto x^3$ has image contained in $G_1^0$ and its restriction to $G_1^0$ is onto. Consider the map sending $x\in G_1^0$ to $(ax)^3 \in G_1^0$. Since $G^0_1$ is central, $(ax)^3 = a^3 x^3$. Now let $y\in G^0_1$ be such that $y^3 = a^3$. Then $ay^{-1}$ has order three and generates a complement $C$ of $G^0_1$ in $G_1$.
Since $G^0_1$ is central we conclude that $G_1$ is the direct product of $C$ and $G^0_1$, hence it is abelian.
\end{pf}

\subsection*{Acknowledgements} We thank Rosario Mennuni and an anonymous referee for many suggestions that helped improve the paper.  


\end{document}